\documentclass[11pt,final]{amsart}
\usepackage{amsmath,amssymb,amsthm,amsfonts,mathrsfs,amsopn}
\usepackage[all]{xy}
\usepackage{dsfont}
\usepackage{hyperref}
\usepackage{color}
\usepackage{esint}
\usepackage{mathtools}
\mathtoolsset{showonlyrefs}
\usepackage{slashed}

\usepackage{tikz-cd}

\usepackage{faktor}

\newcommand{\tnorm}[1]{%
  \left\vert\kern-0.3ex\left\vert\kern-0.3ex\left\vert #1 \right\vert\kern-0.3ex\right\vert\kern-0.3ex\right\vert
}

\usepackage[margin=0.9in]{geometry}

\newcommand{\Abracket}[1]{\left<#1\right>} % angle bracket
\newcommand{\parenthesis}[1]{\left(#1\right)} % round bracket
\newcommand{\braces}[1]{\left\{#1\right\}} % curly bracket

\newcommand{\R}{\mathbb{R}}
\newcommand{\C}{\mathbb{C}}
\newcommand{\T}{\mathbb{T}}

\newcommand{\dd}{\mathop{}\!\mathrm{d}}
\newcommand{\eps}{\varepsilon}
\newcommand{\p}{\partial}

\newcommand{\D}{\slashed{D}}
\newcommand{\pd}{\slashed{\partial}}
\newcommand{\pD}{\mathbf{D}}

\DeclareMathOperator{\Dom}{Dom}
\DeclareMathOperator{\dv}{\dd{vol}}

\DeclareMathOperator{\Eigen}{Eigen}

\DeclareMathOperator{\id}{Id}

\DeclareMathOperator{\Mtp}{Multi} % multiplicity

\DeclareMathOperator{\Spect}{Spect}

\DeclareMathOperator{\up}{up}
\DeclareMathOperator{\low}{down}

\newtheorem{thm}{Theorem}[section]

\newtheorem{lemma}[thm]{Lemma}
\newtheorem{prop}[thm]{Proposition}

\title[Periodic Solutions]{Periodic Solutions for a Nonlinear Dirac Equation with Soler-Type Nonlinearity}

\subjclass[2020]{35J50, 35Q70, 81Q15}

\thanks{The author is supported by Beijing Natural Science Foundation No. 1262018.}

\keywords{nonlinear Dirac equation, periodic solutions, Soler-type nonlinearity}

\author[F. Zhang]{Fuping Zhang}
\address{Fuping Zhang, School of Mathematics and Statistics, Beijing Institute of Technology, Zhongguancun South Street No. 5, 100081 Beijing, P.R. China.}
\email{fuping.zhang@bit.edu.cn}

\begin{document}

\begin{abstract}
We study nonlinear Dirac equations on the three-dimensional torus with subcritical Soler-type nonlinearities. The variational problem is strongly indefinite, while the non-coercive Soler potential prevents direct compactness. We recover the Palais--Smale condition by introducing a small coercive perturbation. Refined spectral estimates and a linking argument yield perturbed solutions with uniform energy bounds. Passing to the limit as the perturbation vanishes gives a nontrivial continuously differentiable periodic solution. This existence result holds for every spectral parameter in an explicit interval below the mass.

\end{abstract}

\maketitle

\section{Introduction}

Nonlinear Dirac equations arise naturally in relativistic quantum mechanics and field theory, describing spin-$\frac{1}{2}$ particles with self-interactions. Among the most studied models is the Soler equation~\cite{BoussaidComech2019Nonlinear,Soler1970Classical}, which involves a nonlinear term of the form $(\bar{\psi}\psi)\psi$~\cite{BalabaneCazenave1988Existence,Balabane1988Existence,Balabane1990Existence, Merle1988Existence}. Such equations exhibit rich mathematical structures, including solitary waves, stationary states, and periodic solutions~\cite{DingLiu2014periodic,EstebanSere1995stationary,Esteban2002overview}. On compact manifolds, nonlinear Dirac equations benefit from the discrete Dirac spectrum, and on flat tori this framework naturally leads to periodic solutions, which in turn suggests the possibility of applying variational methods~\cite{BartschDing2006Solutions,Ding2007variational,DingLiXu2016bifurcation,DingRuf2008Solutions}.

In this paper, we consider the nonlinear Dirac equation on the flat torus $\T^3$:
\begin{align}\label{eq:NDE-spatial}
-\sum_{j=1}^3 i\gamma^0\gamma^j\partial_j\psi + m\gamma^0\psi
= a\psi + 2h(x)\nu |\bar{\psi}\psi|^{\nu-2}\cdot \bar{\psi}\psi\,\gamma^0\psi,
\end{align}
where $m>0$ is the mass, $a\in(0,m)$ is a spectral parameter, $h\in C^\infty(\T^3,\mathbb{R})$ is positive and bounded away from zero, i.e.
\begin{align}
0<h_{\min}:=\min_{x}h(x)\leq h(x)\leq h_{\max}:=\max_{x}h(x),
\end{align}
and $\nu\in(1,\frac{3}{2})$. The~$\gamma^\mu$'s are chosen as follows. The matrix~$\gamma^0$ takes the form
\begin{align}
    \gamma^0= \begin{pmatrix} I_2 & 0 \\ 0 & -I_2\end{pmatrix}
    =\begin{pmatrix}
        1 & & &  \\ & 1 & & \\ & & -1 & \\ & &  & -1
    \end{pmatrix}
\end{align}
and induces a decomposition of the spinors into~$\pm1$ eigenspaces of~$\gamma^0$.
As for the others, we denote by
\begin{align}
    {\sigma^1} = \left( {\begin{array}{*{20}{c}}
			{0}&1 \\
			1&{0}
	\end{array}} \right),\;{\sigma^2} = \left( {\begin{array}{*{20}{c}}
			{0}&{ - i} \\
			i&{0}
	\end{array}} \right),\;{\sigma^3} = \left( {\begin{array}{*{20}{c}}
			1&{0} \\
			{0}&{ - 1}
	\end{array}} \right)
\end{align}
for the Pauli matrices, and the gamma matrices are
\begin{align}
    {\gamma^k} = \left( {\begin{array}{*{20}{c}}
			{0}&{{\sigma^k}} \\
			{{-\sigma^k}}&{0}
	\end{array}} \right),\quad\text{for}\quad k= 1, 2, 3.
\end{align}
Note that they satisfy the following Clifford relations: for~$j,k\in\braces{1,2,3}$,
\begin{align}
    \sigma^j\sigma^k+\sigma^k\sigma^j= 2\delta^{jk}I_2, & &
    \gamma^j\gamma^k + \gamma^k\gamma^j= -2\delta^{jk}I_4.
\end{align}
Moreover,
\begin{align}
    \gamma^0\gamma^k + \gamma^k \gamma^0 =0, & &
    (\gamma^0)^2=I_4.
\end{align}
Then~$\sum_{\mu} i\gamma^\mu \p_\mu$ is the four-dimensional space-time Dirac operator. The nonlinearity is of Soler type and depends on the Lorentz-invariant quantity $\bar\psi\psi=\langle\gamma^0\psi,\psi\rangle$. This quantity may exhibit cancellation and does not control $|\psi|^2$. This makes the variational analysis substantially more delicate than for nonlinearities of the form $|\psi|^{p-2}\psi$~\cite{DingRuf2012Existence,DingWei2008Stationary,ZhangQinZhao2012Multiple,ZhangZhangZhao2018Existence}.

The main difficulty in applying variational methods to~\eqref{eq:NDE-spatial} arises from two sources. First, the linear operator $\pD := -i\sum_j \gamma^0\gamma^j\partial_j + m\gamma^0$ is self-adjoint but indefinite, leading to a strongly indefinite functional. Second, the nonlinearity $|\bar{\psi}\psi|^\nu$ does not control the full $L^{2\nu}$-norm of $\psi$. For the corresponding energy functional, the Palais--Smale condition is not readily available. To overcome this, we introduce a perturbation of the form $-\eps |\psi|^{2\nu}$, which restores the Palais--Smale condition for each $\eps>0$. However, the unperturbed problem $(\eps=0)$ corresponds to the original equation, and one must carefully pass to the limit $\eps\to0^+$.
This strategy is persued in the recent work~\cite{WZ2026splitting,WZ2026Stationary}.
Here we refine the uniform estimates for a special nonlinearity of pure Soler type, namely~$h(x)|\bar{\psi}\psi|^\nu$ for some positive smooth function~$h$ and subcritical exponent~$\nu\in(1,\frac{3}{2})$.

We can now state our main result.
Throughout, the data~$\nu,h,\mathbb{T}^3$ are fixed, and the dependence on them in various constants is suppressed when convenient.
But we should emphasize that most of the constants depend on the torus geometry.  
Let~$\nu^*=\nu/(\nu-1)$ be the H\"older conjugate of $\nu$. 
The constant $\kappa_*(m)$ is defined in Section~\ref{sect:last}.
\begin{thm}\label{thm:main}
There exists a constant
\begin{align}
a_*:=\max\left\{0,\; m-\frac{\kappa_*(m)}{\operatorname{vol}(\mathbb{T}^3)^{1/\nu^*}}
  \left(\frac{h_{\min}}{h_{\max}}\right)^{1/\nu}\right\} \in [0,m)
\end{align}
such that for any $a\in(a_*,m)$, there exists a nontrivial $C^1$ periodic solution to \eqref{eq:NDE-spatial}.
\end{thm}
%  The key step is to show that, under the condition $a>a_*$, the perturbed minimax levels $c_\eps(m,a)$ remain uniformly bounded above by a constant $U(m,a)$ which lies strictly below the blow-up threshold $c_\infty(m,a)$:
% \begin{align}
% c_\eps(m,a)\le U(m,a)<c_\infty(m,a),
%  \qquad 0<\eps\le1.
% \end{align}
% This uniform subthreshold estimate, together with a compactness argument as $\eps\to0^+$, allows us to extract a nontrivial solution of the unperturbed equation. The blow-up threshold $c_\infty(m,a)$ is defined in terms of a limiting blow-up problem. We emphasize that our argument establishes this sufficient condition via the uniform lower bound $\kappa(m,a)\ge\kappa_*(m)$; no necessity of the threshold is claimed.

The proof relies on a combination of tools from critical point theory for strongly indefinite functionals. We work in the fractional Sobolev space $H^{1/2}(\T^3,\mathbb{C}^4)$, equipped with the equivalent norm induced by $|\pD|^{1/2}$. The spectrum of $\pD$ is explicitly computed in Section~\ref{sec:prelim}, and its gap structure plays a crucial role in the compactness estimates. We then construct an admissible pseudo-gradient flow, following the framework developed in~\cite{WZ2026Stationary}, which allows us to define a linking geometry and a minimax level $c_\eps(m,a)$ for the perturbed functionals $J_\eps$. The key step is to show that, under the condition $a>a_*$, the perturbed minimax levels $c_\eps(m,a)$ remain uniformly bounded above by a constant $U(m,a)$ that lies strictly below the blow-up threshold $c_\infty(m,a)$. This uniform separation, together with a subsequential compactness argument for the critical points $\psi_\eps$ as $\eps\to0^+$, allows us to extract a nontrivial limit solving the unperturbed equation.

A central technical contribution of this paper is the refined estimate of the constant $\kappa(m,a)$, defined as the infimum of the $L^{\nu^*}$-norm of a potential $V$ for which there exists a nontrivial spinor $\varphi$ with $\bar\varphi\varphi=0$ and $(\pD-a)\varphi=V\gamma^0\varphi$. We prove a uniform positive lower bound $\kappa(m,a)\ge\kappa_*(m)$ independent of $a\in(0,m)$, using a projection argument onto the kernel of $\pD-m$ and the null-cone condition $\bar\varphi\varphi=0$. This estimate is then combined with an explicit upper bound for the linking level, obtained by testing with a one-dimensional family of spinors, to derive the sufficient condition $a>a_*$.

The structure of the paper is as follows. In Section~\ref{sec:prelim}, we recall the spectral properties of the Dirac operators on $\T^3$ and define the function space. Section~\ref{sect:variational str} introduces the variational setting and the perturbed functionals. In Section~\ref{sect:perturbation}, we verify the Palais--Smale condition for the perturbed problem and discuss the admissible pseudo-gradient flow. Section~\ref{sect:example} contains the definition and estimate of $\kappa(m,a)$ and the subthreshold theorem. Section~\ref{sect:last} is devoted to the proof of the main theorem.

Our results complement the recent work~\cite{WZ2026Stationary} on stationary solutions on compact manifolds, and provide a new existence result for periodic solutions on the flat torus $\T^3$ with Soler-type nonlinearities.

\section{Preliminaries}\label{sec:prelim}

\subsection{The Dirac operators}

We will denote
\begin{align}
    \D= \sum_{k=1}^3 -i\gamma^0\gamma^k\p_k, & &
    \mbox{and} & &
    \pD=\sum_{k=1}^3 -i\gamma^0\gamma^k\p_k+m\gamma^0 .
\end{align}
Note that the coefficients of~$\D$ satisfy
\begin{align}
    (-i\gamma^0\gamma^k)(-i\gamma^0\gamma^j)+ (-i\gamma^0\gamma^j)(-i\gamma^0\gamma^k)=-2\delta^{jk} I_4, \qquad \forall j,k\in \braces{1,2,3}.
\end{align}
Thus~$\D$ satisfies the geometers' convention. In matrix form, we have
\begin{align}
    \D =
    \begin{pmatrix}
        0 & -i\mathbf{\sigma}\cdot\nabla \\ -i\mathbf{\sigma}\cdot\nabla & 0
    \end{pmatrix}
    = \begin{pmatrix}
        0 & \pd \\ \pd & 0
    \end{pmatrix}
\end{align}
where
\begin{align}
    \pd\equiv -i\mathbf{\sigma}\cdot \nabla = \sum_{k=1}^3 -i\sigma^k\p_k
\end{align}
denotes the intrinsic Dirac operator on~$\R^3$ and on~$\mathbb{T}^3$ (equipped
with the trivial spin structure) acting on~$\C^2$-valued functions.

\subsection{Spectrum of the Dirac operators}

Throughout the paper, $\T^3=\mathbb R^3/\Gamma$ is a flat torus
associated with a lattice $\Gamma\subset\mathbb R^3$, equipped with the trivial spin structure. The spectrum of the intrinsic Dirac operator~$\pd$ on~$\T^3$ is explicitly known, see e.g.,~\cite{Friedrich1984zur} and~\cite[Chapter 2.1]{Ginoux2009Dirac}. It is well known that
\begin{align}
    \Spect(\pd_{\T^3})=\braces{\pm2\pi |\zeta^*| \; \mid\;  \zeta^*\in \Gamma^*}.
\end{align}
Let
\begin{align}
\Gamma^*
:=
\left\{
k\in\mathbb R^3:
k\cdot\theta\in\mathbb Z
\ \text{for every }\theta\in\Gamma
\right\}
\end{align}
be the dual lattice of $\Gamma$. Thus, for every $k\in\Gamma^*$, the function
\begin{align}
e_k(x):=
\exp(2\pi i\, k\cdot x),
\qquad x\in\mathbb R^3,
\end{align}
is $\Gamma$-periodic, since
\begin{align}
e_k(x+\theta)
=
e^{2\pi i k\cdot(x+\theta)}
=
e^{2\pi i k\cdot\theta}e_k(x)
=
e_k(x)
\end{align}
for every $\theta\in\Gamma$. The functions $\{e_k\}_{k\in\Gamma^*}$ form an
orthogonal basis of $L^2(\T^3)$, and every spinor
$\psi\in L^2(\T^3,\mathbb C^4)$ admits a Fourier expansion
\begin{align}
\psi(x)
=
\sum_{k\in\Gamma^*}\widehat\psi(k)e^{2\pi i k\cdot x},
\qquad
\widehat\psi(k)\in\mathbb C^4.
\end{align}
Furthermore, the complex multiplicity of~$0$ is~$2^{{\lfloor n/2\rfloor}}=2^{{\lfloor 3/2\rfloor}}=2$. In particular, note that with respect to this trivial spin structure,~$\Spect(\pd)$ is symmetric about~$0$.

Moreover, we have the following spectral information.

\begin{lemma}\label{lemma:spectrum D}~\cite[Lemma 2.2]{WZ2026Stationary}
    The spectrum of~$\pD$ is given by
    \begin{align}
        \Spect(\pD)
        =\braces{\pm\sqrt{\mu^2+ m^2} \; \mid \; \mu\in\Spect(\D)}
    \end{align}
    with multiplicities
    \begin{align}
        \Mtp&(\pD; \sqrt{\mu^2+m^2})= \Mtp(\pD; -\sqrt{\mu^2+m^2}) = \Mtp(\D; \mu),\quad \mbox{ for } \mu\neq 0, \\
        \Mtp&(\pD;m)=\Mtp(\pD;-m)=\frac{1}{2}\Mtp(\D;0)=2.
    \end{align}
\end{lemma}
The details of the proof are contained in~\cite{WZ2026Stationary}, so we omit them here.

\subsection{The working space}

Using the spectral information above, we work on the fractional Sobolev space~$H^{\frac{1}{2}}(\T^3,\C^4)$, whose definition is recalled briefly in the following. In the variational arguments below, $H^{\frac{1}{2}}(\mathbb T^3,\mathbb C^4)$ is regarded as a real Hilbert space, with the real part of the Hermitian pairing as scalar product.

We have seen that~$\pD$ is a self-adjoint elliptic operator and the spectrum consists of nonzero eigenvalues.
Let~$\tau \colon \mathcal{B}(\R)\to \operatorname{Proj}(L^2(\T^3,\C^4))$ be the spectral measure of~$\pD$, where~$\mathcal{B}(\R)$ is the Borel~$\sigma$-algebra of~$\R$.
Then the operator~$\pD$ admits the spectral resolution
\begin{align}
    \pD =\int_\R \lambda \dd\tau(\lambda).
\end{align}

The absolute value of $\pD$, denoted by $|\pD|$, is a non-negative self-adjoint operator defined by
\begin{align}
|\pD| = \int_{\mathbb{R}} |\lambda| \, \dd\tau(\lambda),
\end{align}
with domain $\Dom(|\pD|) := \left\{ \psi \in L^2(\mathbb{T}^3, \mathbb{C}^4) \mid \int_{\mathbb{R}} |\lambda|^2 \, \dd\Abracket{\tau(\lambda)\psi,\psi}_{L^2} < \infty \right\}$.
The operator~$|\pD|^{\frac{1}{2}}$ is similarly defined by
\begin{align}
|\pD|^{1/2} = \int_{\mathbb{R}} |\lambda|^{1/2} \, \dd\tau(\lambda),
\end{align}
with domain
\begin{align}
   \Dom(|\pD|^{1/2}) = \left\{ \psi \in L^2(\mathbb{T}^3, \mathbb{C}^4) \mid \int_{\mathbb{R}} |\lambda| \, \dd\Abracket{\tau(\lambda)\psi,\psi}_{L^2}  < \infty \right\}.
\end{align}

We will work with the space
\begin{align}
    H^{\frac{1}{2}}(\T^3, \mathbb{C}^4)
    \coloneqq \braces{ \psi\in L^2(\T^3;\C^4)\;\mid \; \|\psi\|_{H^{\frac{1}{2}}} \equiv \parenthesis{\|\psi\|_{L^2}^2 + \||\pD|^{\frac{1}{2}}\psi\|_{L^2}^2}^{\frac{1}{2}} <+\infty }.
\end{align}
Note that for each~$\psi\in  H^{\frac12}(\T^3, \mathbb{C}^4)$, we have
  \begin{align}
      \int_{\mathbb{T}^3} ||\pD|^{1/2}\psi|^2 \dv 
      = \langle |\pD|^{1/2}\psi, |\pD|^{1/2}\psi \rangle_{L^2(\T^3, \mathbb{C}^4)}
      = \langle |\pD|\psi, \psi \rangle_{H^{-\frac{1}{2}}\times H^{\frac{1}{2}}} 
      = \int_{\mathbb{R}} |\lambda| \,\dd\Abracket{\tau(\lambda)\psi,\psi}_{L^2} .
  \end{align}
Since the smallest eigenvalue of~$|\pD|$ is~$m>0$, we see that 
\begin{align}
    \|\psi\|_{L^2} \leq \frac{1}{\sqrt{m}}\| |\pD|^{\frac{1}{2}} \psi\|_{L^2}. 
\end{align}
Thus we have an equivalent norm on~$H^{\frac{1}{2}}(\T^3,\C^4)$ given by 
\begin{align}\label{eq:equiv norm}
    \tnorm{\psi}_{H^{1/2}(\T^3,\C^4)}:= \| |\pD|^{\frac{1}{2}} \psi\|_{L^2(\mathbb{T}^3,\C^4)}= \parenthesis{ \langle |\pD|\psi, \psi \rangle_{H^{-\frac{1}{2}}\times H^{\frac{1}{2}}} }^{\frac{1}{2}}.
\end{align}
We remark that for~$\psi\in H^{\frac{1}{2}}(\T^3,\C^4)$, the expression~$\int_{\T^3}\Abracket{\pD\psi,\psi}\dv$ will always be understood in the duality pairing sense.

Using the spectral decomposition of~$\pD$, we define the closed subspaces spanned by eigenspinors of positive [resp. negative] eigenvalues by
\begin{align}
    H^{\frac{1}{2},+}\coloneqq
    \overline{\bigoplus_{\lambda>0} \Eigen(\pD; \lambda)}^{H^{\frac{1}{2}}},
    \quad \mbox{resp.}
    \quad
    H^{\frac{1}{2},-}\coloneqq
    \overline{\bigoplus_{\lambda<0} \Eigen(\pD; \lambda)}^{H^{\frac{1}{2}}}.
\end{align}
Then we have the following orthogonal decomposition
\begin{align}
    H^{\frac{1}{2}}(\mathbb{T}^3, \mathbb{C}^4)
    =H^{\frac{1}{2},+}\oplus H^{\frac{1}{2},-}.
\end{align}
Furthermore, we let
\begin{align}
    P^{\pm} \colon H^{\frac{1}{2}}(\mathbb{T}^3,\C^4) \to H^{\frac{1}{2},\pm}
\end{align}
denote the orthogonal projections onto the subspaces~$H^{\frac{1}{2},\pm}$.
For any~$\psi\in H^{\frac{1}{2}}$, we can decompose it as
\begin{align}
    \psi = P^+\psi + P^-\psi \equiv \psi^+ + \psi^- \in H^{\frac{1}{2},+}\oplus H^{\frac{1}{2},-}.
\end{align}
Note that this is also a~$L^2$ orthogonal projection:
\begin{align}
    \|\psi\|_{L^2}^2 = \|\psi^+\|_{L^2}^2 + \|\psi^-\|_{L^2}^2.
\end{align}

With the above projectors, we have
\begin{align}
    \pD = \pD(P^+ + P^-) = \pD P^+ + \pD P^-, & &\mbox{ and }& &
    |\pD|= \pD P^+ - \pD P^-.
\end{align}

\section{The variational structure}\label{sect:variational str}

The equation~\eqref{eq:NDE-spatial} is the Euler-Lagrange equation of the functional
\begin{align}
    J\colon H^{\frac{1}{2}}(\T^3,\C^4)\to \R
\end{align}
given by
\begin{align}
    J(\psi)
    =\int_{\T^3}  \frac{1}{2}\Abracket{\psi,\pD\psi} -\frac{a}{2}|\psi|^2 - h(x)|\bar{\psi}\psi|^\nu\dv.
\end{align}
In particular, we have
\begin{align}
    \bar{\psi}\psi = \Abracket{\gamma^0\psi,\psi}= |\psi_{\up}|^2 - |\psi_{\low}|^2, & &
    |\psi|^2 = |\psi_{\up}|^2 + |\psi_{\low}|^2 = \sum_{k=1}^4 |\psi^k|^2.
\end{align}
Note that~$\bar{\psi}\psi$ may be small, while~$|\psi|^2$ is large.
This kind of nonlinearity, which frequently arises in various particle models in quantum field theory, is more delicate than nonlinearities of the form~$|\psi|^2$. We seek new periodic solutions by variational methods.

\section{Perturbations}\label{sect:perturbation}

For each~$\eps\in [0,1]$, consider the perturbed functional
\begin{align}
    J_\eps(\psi)
    =& J(\psi)-\eps\int_{\T^3} |\psi|^{2\nu}\dv \\
    =& \int_{\T^3} \frac{1}{2}\Abracket{\psi,\pD\psi}-\frac{a}{2}\Abracket{\psi,\psi}- h(x)|\bar{\psi}\psi|^\nu-\eps|\psi|^{2\nu}\dv.
\end{align}
The Euler-Lagrange equation of~$J_\eps$ is
\begin{align}\label{eq:perturbative Dirac eq}
    \pD\psi -a\psi -2\nu h(x)\bar{\psi}\psi\cdot|\bar{\psi}\psi|^{\nu-2}\gamma^0\psi= 2\nu\eps|\psi|^{2\nu-2}\psi.
\end{align}

Since the nonlinearity ~$F(x,\psi)=h(x)|\bar{\psi}\psi|^\nu$ studied here satisfies conditions (F1)-(F5) from Theorem 1.1 of the reference~\cite{WZ2026Stationary}, the corresponding perturbed functional fulfills both the Palais--Smale condition~\cite{WZ2026Stationary} and the linking geometry~\cite[Lemma 4.1]{WZ2026Stationary}.

In Section 4 of the cited reference~\cite{WZ2026Stationary}, the authors consider the case~$N=1$. In this case, choose 
\begin{align}
0\ne e_H\in K:=\ker(\pD-m),\qquad
\tnorm{e_H}_{H^{1/2}}=1,\qquad
\mathscr E_1:=\operatorname{span}_{\R}\{e_H\}.
\end{align}
Let 
\begin{align}
    \mathscr{C}_1(R)
    \coloneqq \braces{\psi= \psi^- + \lambda e_H\in H^{\frac{1}{2}}(\mathbb{T}^3,\C^4) \mid  \psi^-\in H^{\frac{1}{2},-}, \; \tnorm{\psi^-}_{H^{1/2}}\leq R, \; \lambda \in [0,R]}. 
\end{align}
For~$r\in (0,R)$, consider the sphere in~$H^{\frac{1}{2},+}$ with radius~$r$:
\begin{align}
    \mathscr{S}^+(r)\coloneqq \braces{\psi\in H^{\frac{1}{2},+} \; \mid \; \tnorm{\psi}_{H^{{1}/{2}}}=r }.
\end{align}
Since both sets are infinite-dimensional, the classical finite-dimensional linking argument from~\cite{Ambrosetti2007Nonlinear,Benci1979Critical} does not directly apply. Nevertheless, we will see that they separate the functional~$J_\eps$ in the sense of~\cite{Schechter2021linking}.

\subsection{Verification of the Palais--Smale condition for the perturbed functionals}
Fix~$\eps \in (0,1]$, and let~$(\psi_n)_{n\geq 1}$ be a Palais--Smale sequence for~$J_\eps$ at a level~$c>0$, namely
\begin{align}\label{eq:PS-level}
    J_\eps(\psi_n)=\int_{\T^3} \frac{1}{2}\langle \psi_n, \pD\psi_n\rangle - \frac{a}{2}|\psi_n|^2 - h(x)|\bar{\psi_n}\psi_n|^\nu - \eps|\psi_n|^{2\nu} \dv \to c,
\end{align}
\begin{align}\label{eq:PS:differential}
    \dd J_\eps(\psi_n)= \pD\psi_n - a\psi_n - 2\nu h(x)\bar{\psi_n}\psi_n|\bar{\psi_n}\psi_n|^{\nu-2}\gamma^0\psi_n -2\eps\nu|\psi_n|^{2\nu-2}\psi_n \; \to \; 0
    , \qquad \mbox{ in } H^{-\frac{1}{2}}
\end{align}
as~$n\to +\infty$.

\begin{prop}
    By~\cite[Section 4.1]{WZ2026Stationary}, for each~$\eps\in (0,1]$, the Palais--Smale condition holds for~$J_\eps$. Every Palais--Smale sequence for $J_\eps$ is bounded in $H^{1/2}$, although the bound may depend on $\eps$.
\end{prop}

To construct deformations compatible with the strongly indefinite spectral splitting, we introduce the following class of admissible pseudo-gradient fields~\cite{WZ2026Stationary}. In the following part of the section, ~$\eps\in(0,1]$ will be fixed. A vector field~$W$ on~$H^{\frac{1}{2}}(\T^3,\C^4)$ is called an \emph{admissible pseudo-gradient field} if it has the form
\begin{align}
    W(\psi)= \eta(\psi) \parenthesis{ (P^+ -P^- )\psi - \widetilde{K}(\psi)},
\end{align}
where
\begin{itemize}
    \item[(W1)] $\eta\colon H^{\frac{1}{2}}\to [0,1]$ is a locally Lipschitz function,
    \item[(W2)] $\widetilde{K}\colon H^{\frac{1}{2}}\to H^{\frac{1}{2}}$ is locally Lipschitz and is compact on bounded sets;
    \item[(W3)] $W|_{\p\mathscr{C}_1(R)} = 0$;
    \item[(W4)] On the region~$\braces{\psi\colon \eta(\psi)>0}$, there holds~$\dd J_\eps(\psi)[P^+\psi -P^- \psi -\widetilde{K}(\psi)]\geq 0$;
    \item[(W5)] $\tnorm{W(\psi)}_{H^{1/2}} \leq 1$ for any~$\psi\in H^{\frac{1}{2}}$.
\end{itemize}
This notion is slightly different from the classical one as in~\cite{Ambrosetti2007Nonlinear}.
In particular, we require it to be uniformly bounded but not necessarily comparable to the genuine gradient.

Being bounded and locally Lipschitz, the vector field~$-W$ generates a global flow~$\phi^W_t\colon H^{\frac{1}{2}}\to H^{\frac{1}{2}}$:
\begin{align}
    \begin{cases}
        \frac{\p\phi^W_t(\psi)}{\p t} = - W(\phi^W_t(\psi)), & \quad \forall \psi \in H^{\frac{1}{2}}, \; \forall t\geq 0, \\
        \phi^W_0(\psi)=\psi, & \quad \forall \psi\in H^{\frac{1}{2}}.
    \end{cases}
\end{align}
The property (W3) guarantees that~$\p\mathscr{C}_1(R)$ is fixed along the flow, and (W4) implies that the functional~$J_\eps$ does not increase along the flow.
For this reason we call this flow~$(\phi^W_t)$ an admissible pseudo-gradient flow.
For each~$t\geq 0$, we call~$\phi^W_t$ a time-$t$ map of the admissible pseudo-gradient flow.

Note that (W5) implies that
\begin{align}
    \tnorm{\phi^W_t(\psi)-\psi}_{H^{1/2}}\leq \int_0^{|t|} \tnorm{W(\phi^W_s(\psi))}_{H^{1/2}}\dd{s} \leq |t|.
\end{align}
Thus, there is no blow-up in finite time, and bounded sets are flowed into bounded sets in finite time.

Fix a~$\psi\in H^{\frac{1}{2}}$ and consider a single flow line~$\phi^W_t(\psi)$.
Taking the negative projection of the flow equation, we have
\begin{align}
    P^-\parenthesis{\frac{\p \phi^W_t(\psi)}{\p t}}
    = \eta(\phi^W_t(\psi))P^-\phi^W_t(\psi)
    + \eta(\phi^W_t(\psi))P^-\widetilde{K}(\phi^W_t(\psi)).
\end{align}
It follows that
\begin{align}
    P^-\phi^W_t(\psi)
    = e^{u^W(t,\psi)}P^-\psi + C^W(t,\psi)
\end{align}
with
\begin{align}
    u^W(t,\psi)=\int_0^t \eta(\phi^W_s(\psi))\dd s, \qquad  0\leq u^W(t,\psi)\leq t ,
\end{align}
\begin{align}
    C^W(t,\psi)
    = \int_0^t e^{\int_s^t \eta(\phi^W_\tau(\psi))\dd \tau } \eta(\phi^W_s(\psi))P^-\widetilde{K}(\phi^W_s(\psi))\dd{s}.
\end{align}
We see that for any bounded set~$B\subset H^{\frac{1}{2}}$ and any bounded time interval~$[0,T]$, the map
\begin{align}
    C^W\colon [0,T]\times B \to H^{\frac{1}{2},-}
\end{align}
is continuous and compact.

Let~$\Gamma_\eps$ be the collection of finite compositions of time maps of admissible pseudo-gradient flows:
\begin{align}
    \Gamma_\eps\coloneqq \braces{ \phi^{W_k}_{t_k}\circ \cdots\circ \phi^{W_1}_{t_1} \mid W_j \mbox{ is admissible }, t_j\geq 0, \; k\in\mathbb{N} }.
\end{align}
Setting~$t_j=0$, we see that~$\id\in\Gamma_\eps\neq \emptyset$.
Moreover,~$\Gamma_\eps$ is closed under composition (so it is a semigroup), and any element~$\tilde{\phi}$ is homotopic to the identity map, say via~$\Phi_\eps(t)$:~$\Phi_\eps(0)=\id$,~$\Phi_\eps(1)=\tilde{\phi}$, and for~$t\in [0,1]$,
\begin{itemize}
    \item[(1)] $\Phi_\eps(t)$ fixes~$\p\mathscr{C}_1(R)$ pointwise;
    \item[(2)] $P^-\Phi_\eps(t,\psi)=e^{u(t,\psi)}P^-\psi+ C_{\Phi_\eps}(t,\psi)$ where~$u$ is continuous and bounded, and~$C_{\Phi_\eps}$ is continuous and compact.
\end{itemize}

\begin{lemma}\label{lemma:intersection}~\cite[Lemma 4.6]{WZ2026Stationary}
    For any~$\tilde{\phi}\in \Gamma_\eps$ and~$\Phi_\eps(t)$ as above,
    \begin{align}
        \Phi_\eps(t,\mathscr{C}_1(R))\cap \mathscr{S}^+(r)\neq \emptyset,\qquad \forall t\in [0,1].
    \end{align}
\end{lemma}

The corresponding minimax level is defined by 
\begin{align}
c_\eps(m,a)
:=
\inf_{\Phi_\eps\in\Gamma_\eps}
\sup_{\zeta\in\mathscr C_1(R)}
J_\eps(\Phi_\eps(\zeta)).
\end{align}
We next prove subthreshold compactness by the blow-up analysis.

\section{The pure Soler case and the subthreshold estimate}\label{sect:example}

We now specialize the variational construction to the pure
Soler potential
\begin{align}\label{eq:F explicit}
    F(x,\psi)= h(x)|\bar{\psi}\psi|^{\nu}
\end{align}
where~$h\in C^\infty(\T^3,\R)$ is such that~$0<h_{\min}=\min_{x}h(x)\leq h\leq h_{\max}=\max_{x}h(x)$, and~$\nu\in (1,\frac{3}{2})$.

For~$F$ given in~\eqref{eq:F explicit}, we have
\begin{align}
    \p_\psi F(x,\psi)=2\nu h(x)\bar{\psi}\psi|\bar{\psi}\psi|^{\nu-2}\gamma^0\psi.
\end{align}
\begin{lemma}~\cite[Section 4.3]{WZ2026Stationary}
The number $c_\eps(m,a)$ is a critical value of $J_\eps$. More precisely, there exists $\psi_\eps\ne0$ such that
\begin{align}
    J_\eps (\psi_\eps) = c_\eps(m,a)\in [C_*,\frac{m-a}{2m}R^2], & & 
    \dd J_\eps(\psi_\eps)=0, 
\end{align}
where~$C_*>0$ is independent of $\eps\in(0,1]$.
\end{lemma}
Thus,
\begin{align}
    J_\eps(\psi_\eps)
=\frac12\int_{\T^3}
\langle\psi_\eps,(\pD-a)\psi_\eps\rangle\dv
-\left(\int_{\T^3} h(x)|\bar\psi_\eps\psi_\eps|^{\nu}
+\eps|\psi_\eps|^{2\nu}
\dv\right)= c_\eps(m,a),
\end{align}
\begin{align}
    \partial_\psi J_\eps(\psi_\eps)[\psi_\eps]
=\int_{\T^3}
\langle\psi_\eps,(\pD-a)\psi_\eps\rangle\dv
-2\nu\left(\int_{\T^3} h(x)|\bar\psi_\eps\psi_\eps|^{\nu}
+\eps|\psi_\eps|^{2\nu}
\dv\right)=0.
\end{align}

Let
\begin{align}
K:=\ker(\pD-m).
\end{align}
Recall the normalized eigenspinor $e_H\in K$ fixed above.
\begin{align}
    e_H\in K,\qquad \tnorm{e_H}_{H^{\frac{1}{2}}}=1.
\end{align}
Let~$e=\sqrt{m}e_H$. Then
\begin{align}
\|e\|_{L^2}=1,\qquad
\{\lambda e_H:\lambda\ge0\}
=
\{t e:t\ge0\},
\qquad
\pD e=me,
\qquad
\gamma^0e=e.
\end{align}
In the chosen representation, we take
\begin{align}
e=\binom{e_0}{0},\qquad e_0={\operatorname{vol}(\mathbb T^3)}^{-1/2}\xi,\qquad |\xi|=1,
\end{align}
where $e_0\in\mathbb C^2$, so that $\|e\|_{L^2}=1$.

Let $\psi^-\in H^{1/2,-}$ belong to the negative spectral subspace, so that
$\psi^-\perp e$ in $L^2$. Since $\pD$ is self-adjoint and $\pD e=me$, we have
\begin{align}
\int_{\T^3}
\langle (\pD-a)\psi^-,e\rangle\dv=0.
\end{align}
Moreover, since $\gamma^0e=e$,
\begin{align}
\int_{\T^3}
\langle \gamma^0\psi^-,e\rangle\dv
=
\int_{\T^3}
\langle \psi^-,\gamma^0e\rangle\dv
=
\langle\psi^-,e\rangle_{L^2}
=0.
\end{align}
Therefore all mixed terms in the following $L^2$-integrals vanish. In particular,
\begin{align}
\int_{\T^3}
\langle \psi^-+te,\gamma^0(\psi^-+te)\rangle\dv
&=
t^2+
\int_{\T^3}
\langle\gamma^0\psi^-,\psi^-\rangle\dv.
\end{align}
Likewise,
\begin{align}
\int_{\T^3}
\langle \psi^-+te,(\pD-a)(\psi^-+te)\rangle\dv
=
\int_{\T^3}
\langle\psi^-,(\pD-a)\psi^-\rangle\dv
+(m-a)t^2.
\end{align}

Let~$\psi=\psi^-+te$, $r:=\|\psi^-\|_{L^2}$. Then
    \begin{align}
        \int_{\T^3}\frac{1}{2}\Abracket{\psi, (\pD-a)\psi}\dv
        =&\int_{\T^3} \frac{1}{2}\Abracket{\psi^- + te, (\pD-a)(\psi^-+ te)}\dv \\
        \leq & \frac{m-a}{2}t^2-\frac{m+a}{2}r^2,
    \end{align}
    and
    \begin{align}
        \int_{\T^3}\bar{\psi}\psi\dv
        =\int_{\T^3}\overline{(\psi^- +te)}(\psi^- +te)\dv
        =t^2+\int_{\T^3} \Abracket{\gamma^0\psi^-,\psi^-}\dv
        \geq t^2-r^2.
    \end{align}
    If $r\geq t$, then
    \begin{align}
    J(\psi^-+ te)
    \leq&\frac{m-a}{2} t^2 -\frac{m+a}{2} r^2 - \int_{\T^3} F(x,\psi)\dv \\
    \leq & \frac{m-a}{2} t^2 -\frac{m+a}{2} r^2 \leq -a t^2\leq 0.
    \end{align}

For every admissible pair ~$(\psi^-,t)$ satisfying $r:=\|\psi^-\|_{L^2}<t$, we define
\begin{align}
x:=\frac{r^{2}}{t^{2}}\in[0,1),\qquad r=t\sqrt{x}.
\end{align}
By H\"older inequality, 
\begin{align}
    \int_{\T^3} |\bar{\psi}\psi|^\nu\dv
       \geq  {\operatorname{vol}(\mathbb{T}^3)}^{1-\nu}(t^2-r^2)^{\nu}.
    \end{align}
The preceding estimates yield    
\begin{align}
J(\psi^-+te)\le g_x(t),
\end{align}
where
\begin{align}
g_x(t)
=
\frac{m-a-(m+a)x}{2}\,t^2
-
h_{\min}\operatorname{vol}(\mathbb T^3)^{1-\nu}
(1-x)^\nu t^{2\nu}.
\end{align}
When
\begin{align}
m-a-(m+a)x>0\Leftrightarrow 0\leq x<\frac{m-a}{m+a},
\end{align}
maximizing over $t\ge0$ yields
\begin{align}
\max_{t\ge0}g_x(t)
&=
\frac{\nu-1}{\nu^{\nu^*}}
\frac{
\left(\frac{m-a-(m+a)x}{2}\right)^{\frac{\nu}{\nu-1}}
}{
\left(
h_{\min}\operatorname{vol}(\mathbb T^3)^{1-\nu}
(1-x)^\nu
\right)^{\frac1{\nu-1}}
}\\
&=
\frac{\nu-1}{\nu^{\nu^*}}
\frac{\operatorname{vol}(\mathbb T^3)}
{h_{\min}^{\frac1{\nu-1}}\,2^{\nu^*}}
\left(
\frac{m-a-(m+a)x}{1-x}
\right)^{\frac{\nu}{\nu-1}},\qquad \nu^*=\frac{\nu}{\nu-1}.
\end{align}

Therefore,
\begin{align}
\sup_{\substack{\psi^-\in H^{1/2,-},\,t\ge0\\
\|\psi^-\|_{L^2}<t}}
J(\psi^-+te)
\le
\sup_{0\le x<1}\max_{t\ge0}g_x(t)=\sup_{0\le x<\frac{m-a}{m+a}}
\max_{t\ge0}g_x(t).
\end{align}
For $x\geq\frac{m-a}{m+a}$ we have $\frac{m-a-(m+a)x}{2}\leq0$ and hence
$J(\psi^-+te)\leq0$.

Define
\begin{align}
\ell(x):=\frac{m-a-(m+a)x}{1-x},
\qquad
0\leq x<\frac{m-a}{m+a}.
\end{align}
Then
\begin{align}
\ell'(x)=-\frac{2a}{(1-x)^2}<0.
\end{align}
Therefore $\ell$ is strictly decreasing and its maximum is attained at $x=0$. It follows that
\begin{align}
\sup_{\substack{\psi^-\in H^{1/2,-},\,t\ge0\\
\|\psi^-\|_{L^2}<t}}
J(\psi^-+te)
\leq
\frac{\nu-1}{\nu^{\nu^*}}
\frac{\operatorname{vol}(\mathbb T^3)}
{h_{\min}^{\frac1{\nu-1}}\,2^{\nu^*}}
(m-a)^{\frac{\nu}{\nu-1}}=U(m,a).
\end{align}

We first obtain from the identity deformation that
\begin{align}
c_\eps
\leq
\sup_{\zeta\in\mathscr C_1(R)}
J_\eps(\zeta).
\end{align}
After the normalization of the generator $e$, namely replacing the $H^{1/2}$-normalized vector by the $L^2$-normalized one along the same ray, the linking cone satisfies
\begin{align}
\mathscr C_1(R)
\subset
\left\{
\psi^-+te:
\psi^-\in H^{1/2,-},\ t\geq0
\right\}.
\end{align}

Therefore,
\begin{align}
c_\eps
\leq
\sup_{\substack{
\psi^-\in H^{1/2,-},t\geq0
}}
J_\eps(\psi^-+te).
\end{align}
We now show that the supremum can be restricted to the region
$\|\psi^-\|_{L^2}<t$. Indeed, for~$\psi=\psi^-+te$ with $\|\psi^-\|_{L^2}\geq t$, we have, by the spectral decomposition of $\pD$,
\begin{align}
\frac12\int_{\T^3}
\langle
(\pD-a)\psi,\psi
\rangle\dv
\leq0 .
\end{align}
Since the nonlinear terms are non-positive in the functional $J_\eps$, it follows that
\begin{align}
J_\eps(\psi)
\leq
J(\psi)
\leq0 .
\end{align}
On the other hand, 
\begin{align}
J_\eps(0)=0.
\end{align}
Hence the contribution from the region~$\|\psi^-\|_{L^2}\geq t$ cannot increase the supremum. Consequently,
\begin{align}
c_\eps
\leq
\sup_{\substack{\psi^-\in H^{1/2,-},\,t\ge0\\
\|\psi^-\|_{L^2}<t}}
J_\eps(\psi^-+te).
\end{align}
We obtain
\begin{align}\label{eq:uniform estimate for perturbation}
c_\eps(m,a)
\leq
\sup_{\substack{\psi^-\in H^{1/2,-},\,t\ge0\\
\|\psi^-\|_{L^2}<t}}
J_\eps(\psi^-+te)
&\leq
\sup_{\substack{\psi^-\in H^{1/2,-},\,t\ge0\\
\|\psi^-\|_{L^2}<t}}
J(\psi^-+te)\\
&\leq U(m,a)=
\frac{\nu-1}{\nu^{\nu^*}}
\frac{\operatorname{vol}(\mathbb T^3)}
{h_{\min}^{\frac1{\nu-1}}\,2^{\nu^*}}
(m-a)^{\frac{\nu}{\nu-1}},
\qquad\text{for every }\eps\in(0,1].
\end{align}

We see that
\begin{align}
    \int_{\mathbb{T}^3} h(x)|\bar{\psi}_\eps\psi_\eps|^{\nu}+ \eps|\psi_{\eps}|^{2\nu}\dv=\frac{c_\eps(m,a)}{\nu-1}.
\end{align}
In particular, we have
\begin{align}
\eps
\int_{\T^3}
|\psi_{\eps}|^{2\nu}
\dv
\le
\frac{U(m,a)}{\nu-1},
\end{align}
and
\begin{align}
\int_{\T^3}
h(x)|\overline{\psi_{\eps}}\psi_{\eps}|^\nu
\dv
\le
\frac{U(m,a)}{\nu-1}.
\end{align}

Let~$\nu^*$ be the H\"{o}lder conjugate of~$\nu$, namely~$\frac{1}{\nu}+\frac{1}{\nu^*}=1$, so~$\nu^*=\frac{\nu}{\nu-1}$.
For later convenience we let~$\nu'$ be such that
\begin{align}
    \frac{1}{3} +\frac{1}{\nu^*}=\frac{1}{\nu'}
\end{align}
Since~$\nu\in (1,\frac{3}{2})$, we have~$\nu^*>3$ and~$\nu'=\frac{3\nu}{4\nu-3}\in (\frac{3}{2},3)$.

Fix $m>0$ and $a\in(0,m)$. Define the threshold constant:
\begin{align}\label{the threshold constant}
\kappa(m,a)\coloneqq \inf\left\{
\|V\|_{L^{\nu^*}}: V\in L^{\nu^*}(\mathbb T^3,\mathbb R),
\exists 0\neq\varphi\in W^{1,\nu'}(\T^3,\mathbb C^4), \quad \mbox{s.t.}
(\pD-a)\varphi=V\gamma^0\varphi, \quad  \bar\varphi\varphi=0,
\right\}.
\end{align}
As usual, $\inf\varnothing:=+\infty$.

\begin{lemma}
 For every $\eps_0\in(0,1]$, there exists a constant $C_{\eps_0}>0$ such that
\begin{align}
 \sup_{\eps\in[\eps_0,1]}
 \|\psi_\eps\|_{H^{1/2}}
\le C_{\eps_0}.
\end{align}
\end{lemma}

\begin{proof}

We prove the assertion as follows. For a critical point, we have
\begin{align}
\int_{\T^3}\langle(\pD-a)\psi_\eps,\psi_\eps\rangle\dv
=
2\nu \int_{\T^3}h(x)|\bar{\psi_\eps}\psi_\eps|^{\nu}\dv+2\nu\eps \int_{\T^3}|\psi_\eps|^{2\nu}\dv,
\end{align}
and
\begin{align}
\int_{\T^3}h(x)|\bar{\psi_\eps}\psi_\eps|^{\nu}\dv+\eps \int_{\T^3}|\psi_\eps|^{2\nu}\dv
=
\frac{c_\eps(m,a)}{\nu-1}\le
\frac{U(m,a)}{\nu-1}.
\end{align}
Since $\eps\ge\eps_0$, we get
\begin{align}
\eps_0\|\psi_\eps\|_{L^{2\nu}}^{2\nu}
\le
\eps \int_{\T^3}|\psi_\eps|^{2\nu}\dv
\le
\frac{c_\eps(m,a)}{\nu-1}\le
\frac{U(m,a)}{\nu-1}.
\end{align}
The minimax levels are uniformly bounded on $[\eps_0,1]$, so
\begin{align}
\|\psi_\eps\|_{L^{2\nu}}\le C_{\eps_0}.
\end{align}

Since~$2\nu\in (2,3)$, we have~$\frac{2\nu}{2\nu-1}>\frac{3}{2}$. From the uniform $L^{2\nu}$-bound,
\begin{align}
\left|
h|\bar\psi_\eps\psi_\eps|^{\nu-2}
(\bar\psi_\eps\psi_\eps)\psi_\eps
\right|
+
\eps|\psi_\eps|^{2\nu-1}
\le C|\psi_\eps|^{2\nu-1},
\end{align}
so
\begin{align}
\|(\pD-a)\psi_\eps\|_{L^{\frac{2\nu}{2\nu-1}}}\le C_{\eps_0}.
\end{align}
Since $a\in(0,m)$, $\pD-a$ is invertible, and elliptic estimates give
\begin{align}
\|\psi_\eps\|_{W^{1,\frac{2\nu}{2\nu-1}}}\le C_{\eps_0}.
\end{align}
Because $\frac{2\nu}{2\nu-1}>3/2$,
$W^{1,\frac{2\nu}{2\nu-1}}\hookrightarrow H^{1/2}$,
so
\begin{align}
\sup_{\eps\in[\eps_0,1]}
\|\psi_\eps\|_{H^{1/2}}<\infty.
\end{align}

Thus, if the uniform $L^3$-bound fails, then the blow-up sequence must satisfy $\eps_n\to0$, because on every $[\eps_0,1]$, the $H^{1/2}$-bound implies an $L^3$-bound.

\end{proof}

\begin{thm}
\label{thm:Soler-subthreshold}
Assume that
\begin{align}
U(m,a)<c_\infty(m,a)=\frac{\nu-1}{(2\nu)^{\nu^*}}\cdot\frac{1}{h_{\max}^{\frac{1}{\nu-1}}}\kappa(m,a)^{\nu^*}.
\end{align}
Let $\psi_\eps$ be a critical point of $J_\eps$ at the
minimax level $c_\eps(m,a)$. Then
\begin{align}
\sup_{\eps\in(0,1]}
\|\psi_\eps\|_{H^1(\T^3)}<+\infty.
\end{align}
\end{thm}

\begin{proof}
We first prove that there exists $C=C(\mathbb{T}^3,m,a,\nu)>0$, independent of $\eps\in(0,1]$, such that
\begin{align}
\|\psi_\eps\|_{L^3(\T^3)}\leq C.
\end{align}
Suppose for contradiction that the assertion fails. Then there exists a sequence $\eps_n\to0^+$ with
\begin{align}
\lambda_n :=
\|\psi_{\eps_n}\|_{L^3}
\to+\infty.
\end{align}
Set
\begin{align}
\varphi_n=\frac{\psi_{\eps_n}}{\lambda_n},
\end{align}
so that
\begin{align}
\|\varphi_n\|_{L^3}=1.
\end{align}

The normalized equations
\begin{align}\label{eq:normalized eqn}
    \pD\varphi_n-a\varphi_n=2\nu h(x)\lambda_n^{2\nu-2}|\bar{\varphi_n}\varphi_n|^{\nu-2}(\bar{\varphi_n}\varphi_n)\gamma^0\varphi_n + 2\nu\eps_n |\psi_{\eps_n}|^{2\nu -2 }\varphi_n, \qquad \mbox{ on } \; \T^3.
\end{align}
Define
\begin{align}
V_n
:=
2\nu h(x)\lambda_n^{2\nu-2}
\left|\bar{\varphi_n}\varphi_n\right|^{\nu-2}\bar{\varphi_n}\varphi_n.
\end{align}
Then
\begin{align}
(\pD-a)\varphi_n
=
V_n
\gamma^0\varphi_n
+
R_n,
\end{align}
where
\begin{align}
R_n
:=
2\nu\eps_n
|\psi_{\eps_n}|^{2\nu-2}\varphi_n.
\end{align}
We have
\begin{align}
   \|R_n\|_{L^{\frac{2\nu}{2\nu -1}}}=\|2\nu\eps_n |\psi_{\eps_n}|^{2\nu -2 }\varphi_n \|_{L^{\frac{2\nu}{2\nu -1}}}
    = & \frac{2\nu \eps_n^{\frac{1}{2\nu}} }{\lambda_n}
     \parenthesis{ \int_{\T^3} \left|\eps_n^{\frac{2\nu-1}{2\nu}} |\psi_{\eps_n}|^{2\nu -1} \right|^{\frac{2\nu}{2\nu-1}} \dv   }^{\frac{2\nu-1}{2\nu}}\\
    \leq& \frac{2\nu \eps_n^{\frac{1}{2\nu}} }{\lambda_n}
    \parenthesis{ \int_{\T^3} \eps_n |\psi_{\eps_n}|^{2\nu}\dv}^{\frac{2\nu-1}{2\nu}} \\
    \leq& \frac{2\nu \eps_n^{\frac{1}{2\nu}} }{\lambda_n}\parenthesis{ \frac{U(m,a)}{\nu-1} }^{\frac{2\nu-1}{2\nu}}
\end{align}
thanks to~\eqref{eq:uniform estimate for perturbation}.
 Thus, $R_n\to0
\quad\text{in }L^{\frac{2\nu}{2\nu-1}}$.

\begin{align}
\|V_n\|^{\nu^*}_{L^{\nu^*}}&=(2\nu)^{\nu^*}\lambda_n^{(2\nu-2)\nu^*}\int_{\T^3}h(x)^{\nu^*}|\bar{\varphi_n}\varphi_n|^{(\nu-1)\nu^*}\dv\\
&=(2\nu)^{\nu^*}\lambda_n^{2\nu}\int_{\T^3}h(x)^{\nu^*}|\bar{\varphi_n}\varphi_n|^{\nu}\dv\\
&\leq (2\nu)^{\nu^*}h_{\max}^{\nu^*-1}\int_{\T^3}h(x)|\bar{\psi_{\eps_n}}\psi_{\eps_n}|^{\nu}\dv\\
&\leq \frac{(2\nu)^{\nu^*}h_{\max}^{\nu^*-1}}{\nu-1}\cdot U(m,a).
\end{align}
Thus ${V_n}$ is bounded in $L^{\nu^*}$. By weak compactness, passing to a subsequence,
\begin{align}
V_n\rightharpoonup V
\quad\text{weakly in }L^{\nu^*}.
\end{align}

By H\"{o}lder's inequality, using $\frac1{\nu'}=\frac1{\nu^*}+\frac13$, $\|V_n\|_{\nu^*}\le C$ and $\|\varphi_n\|_{L^3}=1$, since $ \nu'>\frac{2\nu}{2\nu-1}>\frac{3}{2}$ and $\T^3$ has finite volume, we obtain \begin{align}
\|V_n\varphi_n\|_{L^{ \nu'}}\leq C \Rightarrow\|V_n\varphi_n\|_{L^{\frac{2\nu}{2\nu-1}}}\leq C.
\end{align}
By elliptic regularity,
\begin{align}
\|\varphi_n\|_{W^{1,\frac{2\nu}{2\nu-1}}}
&\le
C\left(
\|V_n\varphi_n\|_{L^{\frac{2\nu}{2\nu-1}}}
+\|R_n\|_{L^{\frac{2\nu}{2\nu-1}}}
\right)\\
&\le C.\qquad \frac{2\nu}{2\nu-1}>\frac{3}{2}.
\end{align}
Therefore, after extracting a further subsequence,
\begin{align}
\varphi_n\rightharpoonup\varphi
\quad\text{in }W^{1,\frac{2\nu}{2\nu-1}}.
\end{align}
Since $1<\nu<\frac{3}{2}$, the embedding~$W^{1,\frac{2\nu}{2\nu-1}}(\T^3)\Subset L^{3}(\T^3)$ is compact. Therefore,
\begin{align}
\varphi_n\to\varphi
\quad\text{strongly in }L^3.
\end{align}
We know $\|\varphi\|_{L^3}=1$, so
\begin{align}
\varphi\ne0.
\end{align}
We have 
\begin{align}
\bar\varphi_n\varphi_n=
\lambda_n^{-2}
\bar\psi_{\eps_n}\psi_{\eps_n}.
\end{align}
We now prove the null condition. We have
\begin{align}\label{the null condition}
\int_{\T^3}|\bar{\varphi_n}\varphi_n|^{\nu}\dv&=\frac{\int_{\T^3}|\bar{\psi_{\eps_n}}\psi_{\eps_n}|^{\nu}\dv}{\lambda_n^{2\nu}}\\
&\leq \frac{C}{h_{\min}\lambda_n^{2\nu}}\to0.
\end{align}
Since $2\nu<3$ and $\T^3$ has finite volume, strong convergence in $L^3$ implies strong convergence in $L^{2\nu}$. Therefore,
\begin{align}
\bar\varphi_n\varphi_n\to\bar\varphi\varphi
\quad\text{in }L^\nu.
\end{align}
Combined with \eqref{the null condition}, we know
\begin{align}
\bar\varphi\varphi=0
\quad\text{a.e. on }\T^3.
\end{align}
For any smooth test spinor $\eta$, we have:
\begin{align}
\int_{\T^3}
\langle(\pD-a)\varphi_n,\eta\rangle\dv=
\int_{\T^3}
\langle V_n\gamma^0\varphi_n,\eta\rangle\dv
+
\int_{\T^3}\langle R_n,\eta\rangle\dv.
\end{align}
The left-hand side converges to
\begin{align}
\int_{\T^3}\langle(\pD-a)\varphi,\eta\rangle\dv.
\end{align}
Because $\varphi_n\to\varphi$ strongly in $L^3$ and $V_n\rightharpoonup V$ in $L^{\nu^*}$,
\begin{align}
\int_{\T^3}
\langle V_n\gamma^0\varphi_n-V\gamma^0\varphi,\eta\rangle\dv&=\int_{\T^3}
\langle V_n\gamma^0(\varphi_n-\varphi),\eta\rangle\dv+\int_{\T^3}
\langle (V_n-V)\gamma^0\varphi,\eta\rangle\dv.
\end{align}
The first term tends to zero because
\begin{align}
\|V_n\|_{L^{\nu^*}}\le C,\qquad
\|\varphi_n-\varphi\|_{L^3}\to0.
\end{align}
For the second term,
$\langle\gamma^0\varphi,\eta\rangle\in L^\nu$,
since $L^3\subset L^\nu$ on the finite-volume torus, and therefore weak convergence
$V_n\rightharpoonup V\quad\text{in }L^{\nu^*}$
makes it tend to zero. Thus,
\begin{align}
V_n\gamma^0\varphi_n
\rightharpoonup
V\gamma^0\varphi\quad \text{in the distributional sense}.
\end{align}
Hence
\begin{align}
(\pD-a)\varphi=V\gamma^0\varphi
\end{align}
holds in the weak sense. We can recover the regularity required in the definition of
\begin{align}
\kappa:
V\in L^{\nu^*},\qquad \varphi\in L^3
\quad\Longrightarrow\quad
V\varphi\in L^{\nu'},
\end{align}
and elliptic regularity gives
\begin{align}
\varphi\in W^{1,\nu'}.
\end{align}
Thus, the pair $(\varphi,V)$ is admissible in the definition of $\kappa(m,a)$. By the definition of $\kappa(m,a)$ and weak lower semicontinuity of the norm,
\begin{align}
\kappa(m,a)^{\nu^*}
&\leq
\|V\|_{L^{\nu^*}}^{\nu^*}\\
&\leq
\liminf_{n\to\infty}
\|V_n\|_{L^{\nu^*}}^{\nu^*}\\
&\leq
(2\nu)^{\nu^*}
h_{\max}^{\nu^*-1}
\liminf_{n\to\infty}
\int_{\T^3}
h(x)|\bar\psi_{\eps_n}\psi_{\eps_n}|^\nu\dv\\
&\leq \frac{(2\nu)^{\nu^*}
h_{\max}^{\nu^*-1}}{\nu-1}\liminf_{n\to\infty}c_{\eps_n}.
\end{align}
It follows that
\begin{align}
\liminf_{n\to\infty}c_{\eps_n}(m,a)
&\geq
\frac{\nu-1}{(2\nu)^{\nu^*}}
\frac{1}{h_{\max}^{\frac{1}{\nu-1}}}
\kappa(m,a)^{\nu^*}\\
&=:c_\infty(m,a).
\end{align}

We have
\begin{align}
J_\eps(\psi)=J_0(\psi)-\eps\|\psi\|_{L^{2\nu}}^{2\nu}
\le J_0(\psi).
\end{align}
By the upper-bound estimate established above,
$\sup_{\mathscr C_1(R)}J_0\le U(m,a)$; moreover $J_\eps\le J_0$. For each $\eps>0$, the identity map belongs to the admissible class $\Gamma_\eps$, hence
\begin{align}
c_\eps(m,a)
\le
\sup_{\mathscr C_1(R)}J_\eps
\le U(m,a).
\end{align}
By the hypothesis of the theorem,
\begin{align}
U(m,a)<c_\infty(m,a),
\end{align}
we obtain
\begin{align}
c_\infty(m,a)
\le
\liminf_{n\to\infty}c_{\eps_n}(m,a)
\le
\limsup_{n\to\infty}c_{\eps_n}(m,a)
\le U(m,a)
<c_\infty(m,a),
\end{align}
a contradiction.

Therefore
\begin{align}
\sup_{\eps\in(0,1]}
\|\psi_\eps\|_{L^3}<+\infty.
\end{align}

By the uniform $L^3$-estimate, both of the nonlinear terms in the equation
\begin{align}
(\pD-a)\psi_\eps
=
2\nu h(x)|\overline{\psi_\eps}\psi_\eps|^{\nu-2}
(\overline{\psi_\eps}\psi_\eps)\gamma^0\psi_\eps
+
2\nu\eps|\psi_\eps|^{2\nu-2}\psi_\eps,
\end{align}
are bounded pointwise by $C|\psi_\eps|^{2\nu-1}$. Hence
\begin{align}
\|(\pD-a)\psi_\eps\|_{L^{\frac{3}{2\nu-1}}}
\le C.
\end{align}
Since $1<\nu<3/2$, we have
\begin{align}
\frac32<\frac{3}{2\nu-1}<3.
\end{align}
Elliptic regularity and Sobolev embedding therefore yield
\begin{align}
\|\psi_\eps\|_{W^{1,\frac{3}{2\nu-1}}}\le C,
\qquad
\|\psi_\eps\|_{L^{\frac{3}{2\nu-2}}}\le C.
\end{align}

More generally, suppose that $\psi_\eps$ is bounded in $L^r$ for
some $r\ge3$. Then
\begin{align}
(\pD-a)\psi_\eps
\end{align}
is bounded in $L^{r/(2\nu-1)}$. As long as
$r/(2\nu-1)<3$, elliptic regularity followed by Sobolev embedding gives a
bound in $L^{r_{\mathrm{new}}}$, where
\begin{align}
\frac1{r_{\mathrm{new}}}
=
\frac{2\nu-1}{r}-\frac13.
\end{align}
Since $\nu<3/2$, this iteration strictly increases the integrability
exponent. Consequently, after finitely many steps, we reach an exponent
$r$ satisfying
\begin{align}
r\ge 2(2\nu-1).
\end{align}
At this stage,
\begin{align}
\frac{r}{2\nu-1}\ge2.
\end{align}
Since $\T^3$ has finite volume, $L^\frac{r}{2\nu-1}(\T^3)\hookrightarrow L^2(\T^3)$, hence
\begin{align}
\|(\pD-a)\psi_\eps\|_{L^2}\le C.
\end{align}
If at some stage $r/(2\nu-1)\ge3$, then $r\ge3(2\nu-1)>2(2\nu-1)$, and the iteration stops. The $L^2$-boundedness of $\psi_\eps$ and the elliptic estimate for
$\pD-a$ then imply
\begin{align}
\|\psi_\eps\|_{H^1}\le C,
\end{align}
with $C$ independent of $\eps\in(0,1]$.

\end{proof}

\section{Proof of the main theorem}\label{sect:last}

\begin{proof}[Proof of Theorem~\ref{thm:main}]

We perform the orthogonal decomposition $L^2(\T^3;\mathbb{C}^4) = K \oplus K^\perp$, where $K^\perp$ stands for the orthogonal complement of $K$. For any spinor $w\in K^\perp$, the spectrum of $\pD$ restricted to $K^\perp$ satisfies
\begin{align}
\operatorname{Spec}\big(\pD\big|_{K^\perp}\big) \subset (-\infty,-m] \cup \big[\sqrt{m^2+\mu_1^2},+\infty\big),
\end{align}
with $\mu_1>0$ the minimal positive eigenvalue of the geometric Dirac operator $\D$. The distance from $a$ to $\operatorname{Spec}(\pD|_{K^\perp})$ is uniformly positive for $a\in[0,m]$.

We introduce a fixed $L^2$-orthonormal basis $\{e_j\}_{j=1}^N$ for $K=\ker(\pD-m)$. For $1<p<\infty$, we define
\begin{align}
L^p_\perp
=
\left\{
f\in L^p(\T^3,\C^4):
\int_{\T^3}\langle f,e_j\rangle \dv=0,
\ j=1,\dots,N
\right\}.
\end{align}
The projections $P$ and $Q=I-P$ extend continuously to $L^p(\T^3,\C^4)$, and
\begin{align}
Q:L^p(\T^3,\C^4)\rightarrow L^p_\perp
\end{align}
is bounded. Decompose $\varphi = P\varphi+ Q\varphi$.

Consequently, the operator
\begin{align}
\pD - a: W^{1, \nu'}\cap K^\perp \to L^{ \nu'}_\perp,\qquad  \nu'=\frac{3\nu^*}{3+\nu^*},
\end{align}
is boundedly invertible. The details can be found in~\cite[Lemma 5.5]{WZ2026Stationary}. There exists a constant $C(m,\mathbb{T}^3,\nu)$ independent of $a$, such that
\begin{align}\label{eq:uniform-resolvent}
\|w\|_{L^3} \le C(m,\mathbb{T}^3,\nu) \big\|(\pD - a)w\big\|_{L^{ \nu'}},\quad \forall \quad w\in W^{1, \nu'}\cap K^\perp.
\end{align}

Since $K$ is finite-dimensional and consists of smooth spinors, the projection $P$ is given by
\begin{align}
Pf=\sum_{j=1}^N \left(\int_{\T^3}
\langle f,e_j\rangle\dv\right)e_j.
\end{align}
Hence, for any $1<p<\infty$, H\"{o}lder's inequality yields
\begin{align}
\|Pf\|_{L^p}\le C_p\|f\|_{L^p}.
\end{align}
Consequently,
\begin{align}
\|Qf\|_{L^p}\le C_p\|f\|_{L^p}.
\end{align}
In particular,
\begin{align}
\|Qf\|_{L^{\nu'}}\le C\|f\|_{L^{\nu'}}.
\end{align}

Let $(\varphi,V)$ be admissible in the definition of $\kappa(m,a)$, and
normalize $\|\varphi\|_{L^3}=1$. Applying $Q$ gives
\begin{align}
(\pD-a)Q\varphi
=Q(V\gamma^0\varphi)
\end{align}
and applying the elliptic estimate for $\pD-a$ on $K^\perp$, we obtain
\begin{align}
\|Q\varphi\|_{L^3}
&\le C\|Q(V\gamma^0\varphi)\|_{L^{\nu'}} \\
&\le
C(m,\mathbb{T}^3,\nu)\|V\gamma^0\varphi\|_{L^{ \nu'}}\\
&\le
C(m,\mathbb{T}^3,\nu)\|V\|_{L^{\nu^*}}\|\varphi\|_{L^3}\\
&=C(m,\mathbb{T}^3,\nu)\|V\|_{L^{\nu^*}}.
\end{align}

Suppose a spinor $\varphi\in W^{1, \nu'}(\T^3;\mathbb{C}^4)$ satisfies $\bar{\varphi}\varphi=0$ almost everywhere and $\varphi\not\equiv 0$. Since every element of $K$ is a constant spinor of the form $\binom{k_0}{0}$, we may write
$P\varphi=\binom{k_0}{0}$ and split $\varphi = \begin{pmatrix}\phi\\\chi\end{pmatrix}$. The null-cone condition reads
\begin{align}
\bar{\varphi}\varphi = |\phi|^2 - |\chi|^2 = 0 \implies |\phi| = |\chi| \quad \text{a.e.}
\end{align}
For arbitrary constant spinor $l=\begin{pmatrix}
        l_0 \\ 0
    \end{pmatrix}\in K$, we obtain
\begin{align}
|\varphi - l|^2 = |\phi - l_0|^2 + |\chi|^2 = |\phi - l_0|^2 + |\phi|^2 \ge \tfrac12 |l_0|^2.
\end{align}
Thus,
\begin{align}
\|\varphi - l\|_{L^3} \ge \tfrac{1}{\sqrt{2}} \|l\|_{L^3}.
\end{align}
Combining with the triangle inequality $\|\varphi - l\|_{L^3} \ge \big|\|\varphi\|_{L^3} - \|l\|_{L^3}\big|$ and setting~$s = \|l\|_{L^3}$, we obtain
\begin{align}
\operatorname{dist}_{L^3}(\varphi,K)
\ge
\min_{s\ge0}
\max\left\{
\frac{s}{\sqrt2},|1-s|
\right\}
=\sqrt2-1.
\end{align}
Since $Q\varphi = \varphi - P\varphi$, we have
\begin{align}
\sqrt2-1
\le
\operatorname{dist}_{L^3}(\varphi,K)
\le
\|\varphi-P\varphi\|_{L^3}
=\|Q\varphi\|_{L^3}.
\end{align}
Consequently, we obtain
\begin{align}
\kappa(m,a)
\geq
\frac{\sqrt2-1}{C(m,\nu,\mathbb{T}^3)}
=:\kappa_*(m)>0
\end{align}
uniformly for $a\in(0,m)$.

We obtain
\begin{align}
c_\infty(m,a)
&=\frac{\nu-1}{h_{\max}^{\frac{1}{\nu-1}}}\cdot\left(\frac{1}{2\nu}\right)^{\nu^*}\kappa(m,a)^{\nu^*}\\
&\geq
\frac{\nu-1}{(2\nu)^{\nu^*}}\cdot\frac{1}{h_{\max}^{\frac{1}{\nu-1}}}\cdot\kappa_*(m)^{\nu^*}.
\end{align}
Choose~$a$ sufficiently close to $m$ so that
\begin{align}
U(m,a)&=\frac{\nu-1}{\nu^{\nu^*}}\cdot
\frac{\operatorname{vol}(\mathbb{T}^3)}{{h_{\min}^{\frac{1}{\nu-1}}}2^{\nu^*}}
\cdot(m-a)^{\frac{\nu}{\nu-1}}\\
&<\frac{1}{h_{\max}^{\frac{1}{\nu-1}}}\cdot\frac{\nu-1}{(2\nu)^{\nu^*}}\cdot\kappa_*(m)^{\nu^*}\leq c_\infty(m,a)\\
&\Rightarrow m-a< \frac{\kappa_*(m)}{\operatorname{vol}(\mathbb{T}^3)^{1/{\nu^*}}}\cdot\left(\frac{h_{\min}}{h_{\max}}\right)^{\frac{1}{\nu}}.
\end{align}
We may therefore define
\begin{align}
a_*
:=\max\left\{
0,
m-\frac{\kappa_*(m)}{\operatorname{vol}(\mathbb{T}^3)^{1/{\nu^*}}}\cdot\left(\frac{h_{\min}}{h_{\max}}\right)^{\frac{1}{\nu}}
\right\}.
\end{align}
Then for any
\begin{align}
a\in(a_*,m),
\end{align}
we have
\begin{align}
U(m,a)<c_\infty(m,a).
\end{align}

By Theorem \ref{thm:Soler-subthreshold}, we have $\sup_{0<\eps\le1}
\|\psi_\eps\|_{H^1}\le C$. Take any sequence $\eps_n\to0^+$. By weak compactness, there exists a subsequence and $\psi_0\in H^1$ such that
\begin{align}
\psi_{\eps_n}\rightharpoonup\psi_0
\quad\text{weakly in }H^1.
\end{align}
The Rellich compact embedding yields that, up to a further subsequence if necessary,
\begin{align}\label{strongly}
\psi_{\eps_n}\to\psi_0 
\quad\text{strongly in }H^{1/2}, 
\quad\text{and strongly in }L^q,
\quad 1\le q<6.
\end{align}
In particular,
\begin{align}
\psi_{\eps_n}\to\psi_0
\quad\text{strongly in }L^{2\nu}
\end{align}
and
\begin{align}
\bar\psi_{\eps_n}\psi_{\eps_n}
\to
\bar\psi_0\psi_0
\quad\text{strongly in }L^\nu.
\end{align}

From $\sup_{0<\eps\le1}
\|\psi_\eps\|_{H^1}\le C$ and Sobolev embedding,
\begin{align}
\|\psi_{\eps_n}\|_{L^{2(2\nu-1)}}\le C,
\end{align}
since $2(2\nu-1)<6$ holds for $1<\nu<2$. Then
\begin{align}
\left\|
\eps_n|\psi_{\eps_n}|^{2\nu-2}
\psi_{\eps_n}
\right\|_{L^2}
&=
\eps_n
\|\psi_{\eps_n}\|_{L^{2(2\nu-1)}}^{2\nu-1}
\le
C\eps_n
\to0.
\end{align}
Thus the perturbation term tends to zero in $L^2$, hence also in $H^{-1/2}$.

We define
\begin{align}
\mathcal N(z):=|\bar z z|^{\nu-2}(\bar z z)\gamma^0z.
\end{align}
From $|\bar z z|\le C|z|^2$, we have
\begin{align}
|\mathcal N(z)|\le C|z|^{2\nu-1}.
\end{align}
By \eqref{strongly} and standard continuity of Nemytskii maps,
\begin{align}
\mathcal N(\psi_{\eps_n})
\to
\mathcal N(\psi_0)\quad\text{strongly in}\quad L^{\frac{2\nu}{2\nu-1}}.
\end{align}
Passing to the limit,
\begin{align}
(\pD-a)\psi_0=2\nu h(x)
\mathcal N(\psi_0)
\end{align}
holds in $H^{-1/2}$. In other words, $\psi_0$ is a weak solution of the original unperturbed equation.

By the uniform lower bound established above,
\begin{align}
\liminf_{n\to\infty}
J_{\eps_n}(\psi_{\eps_n})
\ge C_*>0.
\end{align}
We claim that
\begin{align}
J(\psi_0)
=\lim_{n\to\infty}
J_{\eps_n}(\psi_{\eps_n}).
\end{align}
Indeed, by strong convergence in $H^{1/2}$ and continuity of the quadratic form
\begin{align}
\left\langle
(\pD-a)\psi_{\eps_n},\psi_{\eps_n}
\right\rangle_{H^{-1/2}\times H^{1/2}}
\rightarrow
\left\langle
(\pD-a)\psi_0,\psi_0
\right\rangle_{H^{-1/2}\times H^{1/2}}.
\end{align}
Moreover, as~$2\nu<3$,
\begin{align}
\int_{\T^3}
h(x)
|\overline{\psi_{\eps_n}}
\psi_{\eps_n}|^\nu
\dv
\rightarrow
\int_{\T^3}
h(x)
|\overline{\psi_0}\psi_0|^\nu
\dv.
\end{align}
Finally, since $\eps_n\rightarrow0$ and the sequence
${\psi_{\eps_n}}$ is bounded in $H^{1/2}$, hence bounded in
$L^{2\nu}$, we obtain
\begin{align}
\eps_n
\int_{\T^3}
|\psi_{\eps_n}|^{2\nu}\dv
\rightarrow0 .
\end{align}
Combining the above three convergences gives
\begin{align}
J(\psi_0)
=\lim_{n\to\infty}
J_{\eps_n}(\psi_{\eps_n})\geq C_*>0.
\end{align}
In particular,~$\psi_0\neq 0$.

For the regularity of~$\psi_0$, since $\psi_0\in H^1$, we have $\psi_0\in L^6$. Therefore
\begin{align}
|\mathcal N(\psi_0)|
\le
C|\psi_0|^{2\nu-1}
\in
L^{\frac6{2\nu-1}}.
\end{align}
For $1<\nu<3/2$,
\begin{align}
\frac6{2\nu-1}>3.
\end{align}
Elliptic regularity gives
\begin{align}
\psi_0\in
W^{1,\frac6{2\nu-1}}.
\end{align}
Because the exponent is strictly larger than 3, the Morrey embedding yields~$\psi_0\in C^{0,\alpha}$ for some~$\alpha\in(0,1)$. 
The real-valued map $s\mapsto s|s|^{\nu-2}$ is H\"older continuous for $\nu>1$.
Thus~$2\nu h(x)\mathcal{N}(\psi_0)\in C^{0,\sigma}$
for some~$\sigma\in (0,1)$, and the Schauder theory then implies~$\psi_0\in C^{1,\sigma}(\T^3)$.
\end{proof}

\textbf{Conflict of Interest.} The author has no conflict of interest.

\textbf{Data Availability Statement.} Data sharing is not applicable because no datasets were generated or analyzed in this study.

\bibliographystyle{plainurl}
\bibliography{nonlinearDiracEquation}

\end{document}